\date{}
\title{Lower bounds on geometric Ramsey functions\thanks{Research
supported by the  ERC Advanced Grant No.~267165 and by ETH Zurich.
}}
\newif\ifcmts
\cmtstrue

\author{
{\sc Marek Eli\'a\v{s}\thanks{Partially supported by the Charles University
Grants GAUK 1262213 and SVV-2013-267313.}}\\
   {\footnotesize Department of Applied Mathematics}\\[-1.5mm]
   {\footnotesize  Charles University, Malostransk\'{e} n\'{a}m. 25}\\[-1.5mm]
{\footnotesize  118~00~~Praha~1,
   Czech Republic}
\and
{\sc Ji\v{r}\'{\i} Matou\v{s}ek}\thanks{Partially
supported by the project CE-ITI (GACR P202/12/G061) of the Czech Science
Foundation.}
\\
   {\footnotesize Department of Applied Mathematics}\\[-1.5mm]
   {\footnotesize  Charles University, Malostransk\'{e} n\'{a}m. 25}\\[-1.5mm]
{\footnotesize  118~00~~Praha~1,
   Czech Republic, and}\\
{\footnotesize    Institute of  Theoretical Computer Science}\\[-1.5mm]
{\footnotesize    ETH Zurich,
      8092 Zurich, Switzerland}
\and
{\sc Edgardo Rold{\'a}n-Pensado}
\\
{\footnotesize    Institute of  Theoretical Computer Science}\\[-1.5mm]
{\footnotesize    ETH Zurich,
      8092 Zurich, Switzerland}
\and
{\sc Zuzana Safernov{\'a}\footnotemark[2]\textsuperscript{~,}\thanks{Partially
supported by the project CE-ITI (GACR P202/12/G061) of the Czech Science
Foundation.}}
\\
   {\footnotesize Department of Applied Mathematics and} \\[-1.5mm]
   {\footnotesize Computer Science Institute}\\[-1.5mm]
   {\footnotesize  Charles University, Malostransk\'{e} n\'{a}m. 25}\\[-1.5mm]
{\footnotesize  118~00~~Praha~1,
   Czech Republic}
}

\documentclass[11pt]{article}

\usepackage{a4wide}
\usepackage{latexsym}
\usepackage{amssymb,amsmath,amsthm}
\usepackage{graphicx}
\usepackage{url}
\usepackage{hyperref}
\usepackage{color}

\newtheorem{theorem}{Theorem}[section]

\newtheorem{prop}[theorem]{Proposition}

\newtheorem{lemma}[theorem]{Lemma}
\newtheorem{corol}[theorem]{Corollary}

\newcommand{\heading}[1]{\vspace{1ex}\par\noindent{\bf #1}}

\makeatletter
\newcommand{\ProofEndBox}{{\ifhmode\unskip\nobreak\hfil\penalty50 \else
          \leavevmode\fi\quad\vadjust{}\nobreak\hfill$\Box$
            \finalhyphendemerits=0 \par}}
\makeatother
\newcommand{\proofend}{\ProofEndBox\smallskip}

\newcommand{\R}{{\mathbb{R}}}

\newcommand\eps{\varepsilon}

\newcommand{\sgn}{\mathop {\rm sgn}\nolimits}

\newcommand\balpha{{\boldsymbol{\alpha}}}
\newcommand\bbeta{{\boldsymbol{\beta}}}

\DeclareMathOperator{\OT}{OT}
\newcommand{\COT}{\OT^*}

\newcommand\sdelta{\overline{\delta}}

\def\:{\colon}

\newcommand{\alterdef}[1]{\!\left\{\!\!\begin{array}{ll}
                                   #1 \end{array}  \right. }

\long\def\onefigure#1#2{
\begin{figure*}[tbp]
\begin{center}
#1
\end{center}
\caption{#2}
\end{figure*}
}

\def\immediateFigure#1{%
\smallskip\begin{center}#1\end{center}\smallskip }

\newcommand{\labfig}[2]  
{\onefigure{\mbox{\includegraphics{Figures/#1}}}{\label{f:#1} #2} }

\newcommand{\labfigw}[3]  
{\onefigure{\mbox{\includegraphics[width=#2]{Figures/#1}}}{\label{f:#1} #3}}

\newcommand{\immfig}[1]  
{\immediateFigure{\mbox{\includegraphics{Figures/#1}}}}

\newcommand{\immfigw}[2] 
{\immediateFigure{\mbox{\includegraphics[width=#2]{Figures/#1}}}}

\DeclareMathOperator{\twr}{twr}
\newcommand\leqone{<_1}

\ifcmts
\newcommand{\marrow}{\marginpar{\boldmath$\longleftarrow$}}
\newcommand{\jirka}[1]{\ifhmode\newline\fi\marrow \textsf{*** (JIRKA: ) #1\newline}}
\newcommand{\marek}[1]{\ifhmode\newline\fi\marrow \textsf{*** (marek: ) #1\newline}}
\newcommand{\edgardo}[1]{\ifhmode\newline\fi\marrow \textsf{*** (Edgardo: ) #1\newline}}
\newcommand{\zuzka}[1]{\ifhmode\newline\fi\marrow \textsf{*** (Zuzka: ) #1\newline}}

\else
\newcommand{\marrow}{}
\newcommand{\jirka}[1]{}
\newcommand{\marek}[1]{}
\newcommand{\edgardo}[1]{}

\fi

\begin{document}

\maketitle

\begin{abstract} 
We continue a sequence of recent works studying Ramsey functions for
semialgebraic predicates in $\R^d$. A \emph{$k$-ary semialgebraic
predicate} $\Phi(x_1,\ldots,x_k)$ on $\R^d$
is a Boolean combination of polynomial equations and inequalities in the
$kd$ coordinates of $k$ points $x_1,\ldots,x_k\in\R^d$. 
A sequence $P=(p_1,\ldots,p_n)$
of points in $\R^d$ is called \emph{$\Phi$-homogeneous} 
if either $\Phi(p_{i_1},
\ldots,p_{i_k})$ holds for all choices $1\le i_1<\cdots<i_k\le n$,
or it holds for no such choice. The Ramsey function $R_\Phi(n)$ is 
the smallest $N$ such that every point sequence of length $N$
contains a $\Phi$-homogeneous subsequence of length~$n$.

 Conlon, Fox, Pach, Sudakov, and Suk constructed the first examples of
semialgebraic predicates with the Ramsey function bounded from below
by a tower function of arbitrary height: for every $k\ge 4$, they exhibit
a $k$-ary $\Phi$ in dimension $2^{k-4}$ with $R_\Phi$  bounded below by a tower of height $k-1$.  We reduce the 
dimension in their construction, obtaining a $k$-ary semialgebraic 
predicate $\Phi$ on $\R^{k-3}$ with $R_\Phi$ bounded below by a tower of height 
$k-1$. 

We also provide a natural geometric Ramsey-type theorem with a large
Ramsey function. We call a point sequence $P$ in $\R^d$ \emph{order-type
homogeneous} if all $(d+1)$-tuples in $P$ have the same orientation.
Every sufficiently long point sequence in general position
in $\R^d$ contains an order-type homogeneous
subsequence of length $n$, and the corresponding Ramsey function
has recently been studied in several papers. Together with
a recent work of B\'ar\'any, Matou\v{s}ek, and P\'or,
our results imply a tower function of $\Omega(n)$
of height $d$ as a lower bound, matching an upper bound by Suk
up to the constant in front of~$n$.
 
\end{abstract}

\section{Introduction}

\heading{Ramsey's theorem and the classical Ramsey function.}
A classical and fundamental theorem of Ramsey claims that for every $n$
there is a number $N$ such that for every coloring of 
the edge set of the complete graph $K_N$ on
$N$ vertices there is a \emph{homogeneous} subset of $n$ vertices,
meaning that all edges in the complete subgraph induced by these
$n$ vertices have the same color. More generally, for every $k$ and $n$
there exists $N$ such that if the set of all $k$-tuples of elements
of an $N$-element set $X$ is colored by two colors, then there
exists an $n$-element homogeneous $Y\subseteq X$, with all $k$-tuples from $Y$
having the same color. Let $R_k(n)$ stand for the smallest $N$
with this property.


Considering $k$ fixed and $n$ large,
the best known lower and upper bounds for the Ramsey function $R_{k}(n)$ are
of the form\footnote{We employ the usual asymptotic notation
for comparing functions: $f(n)=O(g(n))$ means that
$|f(n)|\le C|g(n)|$ for some $C$ and all $n$, where $C$ may depend
on parameters declared as constants (in our case on $k$);
$f(n)=\Omega(g(n))$ is equivalent to $g(n)=O(f(n))$;
and $f(n)=\Theta(g(n))$ means that both $f(n)=O(g(n))$
and $f(n)=\Omega(g(n))$.}
$R_2(n)=2^{\Theta(n)}$ and, for $k\ge 3$,
\[
\twr_{k-1}(\Omega(n^2))\le R_{k}(n) \le \twr_{k}(O(n)),
\]
where the tower function $\twr_k(x)$ is defined by $\twr_1(x) = x$
 and $\twr_{i+1} (x) = 2^{\twr_i (x)}$. 

A widely believed, and probably very difficult, conjecture
of Erd\H{o}s and Hajnal asserts that the upper bound is essentially 
the truth. This  is supported by known bounds for more than two colors,
where the lower bound for $k$-tuples is also a tower
of height $k$; see Conlon, Fox, and
Sudakov \cite{conlon-al} for a recent improvement
and more detailed overview of the known bounds.


\heading{Better Ramsey functions for geometric Ramsey-type results. }
Ramsey's theorem can be used to establish many geometric Ramsey-type results
concerning configurations of points, or of other geometric objects, in $\R^d$.
The first two examples, which up until now remain among the most
significant and beautiful ones, come from a 1935 paper of
Erd\H{o}s and Szekeres \cite{es-cpg-35}.

The first one asserts that every sufficiently long sequence
$(x_1,\ldots,x_N)$ of real numbers contains a subsequence
$(x_{i_1},x_{i_2},\ldots,x_{i_n})$, $i_1<i_2<\cdots<i_n$, that
is either increasing, i.e., $x_{i_1}< x_{i_2}<\cdots<x_{i_n}$,
or nonincreasing, i.e.,  $x_{i_1}\ge x_{i_2}\ge\cdots\ge x_{i_n}$.

Ramsey's theorem for $k=2$ yields the bound $N\le R_2(n)\le \twr_2(O(n))$
(color a pair $\{i,j\}$, $i<j$, red if $x_i<x_j$ and blue 
if $x_i\ge x_j$), but the result is known to hold with $N=(n-1)^2+1$,
an exponential improvement over $R_2(n)$.

For the second of the two Erd\H{o}s--Szekeres theorems mentioned above,
we consider a sequence $P=(p_1,p_2,\ldots,p_N)$ of points in the plane;
for simplicity, we assume that the $p_i$ are in general position
(no three collinear). If $N$ is sufficiently large, then there is
a subsequence $(p_{i_1},\ldots,p_{i_n})$, $i_1<i_2<\cdots<i_n$,
forming the vertex set of a convex $n$-gon, enumerated clockwise
or counterclockwise. 

This time Ramsey's theorem yields $N\le R_3(n)\le\twr_3(O(n))$,
 by coloring a triple
$\{i,j,k\}$, $i<j<k$, red if $p_i,p_j,p_k$ appear clockwise around the
boundary of their convex hull, and blue otherwise. Again, the optimal
bound is one exponential better, of order $2^{\Theta(n)}$.

It is natural to ask, what is special about the two-colorings of 
pairs or triples in the above two examples, what makes the Ramsey
functions here considerably smaller, compared to arbitrary colorings?

One kind of a combinatorial condition for two-colorings 
of $k$-tuples implying such improved bounds was given by
Fox, Pach, Sudakov, and Suk \cite{FoxPachSudSuk}, and another 
by the first two authors  \cite{highes-aim}; both of them include
the two Erd\H{o}s--Szekeres results as special cases.
However, a considerably more general, and probably more interesting,
reason for the better Ramsey behavior of these geometric examples
is that the colorings are ``algebraically defined''; more precisely,
they are given by \emph{semialgebraic predicates}.

\heading{Upper bounds for semialgebraic colorings. }
Let $x_1,\ldots,x_k$ be points in $\R^d$, with $x_{i,j}$ denoting
the $j$th coordinate of $x_i$; we regard the $x_{i,j}$ as variables.
A \emph{$k$-ary $d$-dimensional semialgebraic predicate} $\Phi(x_1,\ldots,x_k)$
 is a Boolean combination of polynomial equations and inequalities 
in the $x_{i,j}$. More explicitly, there are a Boolean formula
$\phi(X_1,\ldots,X_t)$ in  Boolean variables $X_1,\ldots,X_t$ and
polynomials $f_1,\ldots,f_t$ in the variables
$x_{i,j}$, $1\le i\le k$, $1\le j\le d$, such that $\Phi(x_1,\ldots,x_k)=
\phi(A_1,\ldots,A_t)$, where $A_\ell$ is true if
$f_\ell(x_{1,1},\ldots,x_{k,d})\ge 0$ and false otherwise.

We call a sequence $(p_1,\ldots,p_n)$ of points in $\R^d$
{\em $\Phi$-homogeneous} if either $\Phi(p_{i_1},
\ldots,p_{i_k})$ holds for every choice $1\le i_1<\cdots<i_k\le n$,
or it holds for no such choice. The Ramsey function $R_\Phi(n)$ is 
the smallest $N$ such that every point sequence of length $N$
contains a $\Phi$-homogeneous subsequence of length~$n$.

The following general upper bound was first proved by 
Alon, Pach, Pinchasi, Radoi\v{c}i\'c, and Sharir \cite{apprs-semialg} for $k=2$,
and then generalized by Conlon, Fox, Pach, Sudakov, and Suk \cite{cfpss-semialg} for $k\ge 3$:

\begin{theorem}[\cite{apprs-semialg,cfpss-semialg}]\label{t:ub}
 For every $d$, $k$, and a $k$-ary 
$d$-dimensional semialgebraic predicate $\Phi$,
\[
R_\Phi(n)\le\twr_{k-1}(n^C),
\]
where $C$ is a constant depending on $d,k,\Phi$.\footnote{Actually,
the constant $C$ depends on $\Phi$ only through its  \emph{description
complexity}, which Conlon et al.\ define as $\max(m,D)$,
where $m$ is the number of polynomials occurring in $\Phi$ and
$D$ is the maximum degree of these polynomials. Thus, the bound
does not depend on the magnitude of the coefficients in the polynomials.}
\end{theorem}

Thus, the Ramsey function for $k$-ary semialgebraic predicates is
bounded above by a tower one lower than the ``combinatorial''
Ramsey function $R_k(n)$. Let us note that for the case of increasing
or nonincreasing subsequences ($k=2$, $d=1$) and subsequences in
convex position ($k=3$, $d=2$) as above, Theorem~\ref{t:ub} yields
somewhat weak bounds, namely, $n^{O(1)}$ and $2^{n^{O(1)}}$
instead of $n^2$ and $2^{O(n)}$, respectively, but still in the right
range.

By very different methods, Bukh and the second author \cite{BukhMat}
obtained a doubly exponential upper bound for all one-dimensional
semialgebraic predicates, for arbitrary $k$:

\begin{theorem}[\cite{BukhMat}] \label{t:bm}
For every $1$-dimensional
semialgebraic predicate $\Phi$ there is a constant $C$ such that
$R_\Phi(n)\le \twr_3(Cn)$.
\end{theorem}

This opens an interesting possibility, namely, that the Ramsey function
of $d$-dimensional semialgebraic predicates might be bounded
by a tower whose height depends only on $d$ (and not on $k$),
but currently this question is wide open. But certainly it
makes it interesting to study the dependence of the Ramsey function
on the dimension.

\heading{Lower bounds. } The classical Erd\H{o}s--Szekeres result
on subsequences in convex position
\cite{es-cpg-35} supplies a lower bound of $2^{\Omega(n)}=\twr_2(\Omega(n))$
in the setting of Theorem~\ref{t:ub} for $k=3$ and $d=2$.
The first two authors  \cite{highes-aim} constructed a
reasonably natural\footnote{By a ``natural'' predicate we mean here
one that has a clear geometric meaning and
seems reasonable to study in its own right, not only as a lower-bound
example for a general result. In the case of \cite{highes-aim},
assuming that the considered four points $x_1,\ldots,x_4$ are
numbered in the order of increasing first coordinates, the predicate
asserts that $x_4$ lies above the graph of the unique
quadratic polynomial passing through
$x_1,x_2,x_3$.}
 $4$-ary planar semialgebraic $\Phi$ with $R_\Phi(n)\ge
\twr_3(\Omega(n))$. This shows that for $k\le 4$, the height
of the tower in Theorem~\ref{t:ub} is optimal in terms of~$k$.

For $d=1$, \cite{BukhMat} provided a one-dimensional $5$-ary $\Phi$
with $R_\Phi(n)\ge \twr_3(\Omega(n))$, matching Theorem~\ref{t:bm}.
Conlon et al.~\cite{cfpss-semialg} improved the arity to~$4$,
which is optimal in view of Theorem~\ref{t:ub}. 

Moreover, they obtained a lower bound almost matching Theorem~\ref{t:ub}
for an arbitrary $k$. Namely, for every $k\ge 4$ they constructed
a $d$-dimensional $k$-ary semialgebraic predicate $\Phi$
such that $R_\Phi(n) \ge \twr_{k-1}(\Omega(n))$. 
However, the dimension $d$ in their construction is
large:  $d = 2^{k-4}$.

\heading{A stronger lower bound. } 
In this paper we first modify (and simplify) the
lower bound construction of Conlon et al.~\cite{cfpss-semialg},
obtaining examples in considerably lower dimension.

\begin{theorem}\label{t:lb}
For every $d \ge 2$ there is a $d$-dimensional
semialgebraic predicate $\Phi$ of arity $k=d+3$ 
such that 
\[ R_\Phi(n) \ge \twr_{k-1}(\Omega(n)).
\]
\end{theorem}

The proof is given in Section~\ref{s:lower}.
In view of Theorem~\ref{t:bm}, the dependence of the tower height
on the dimension in this result might even be optimal.




\heading{Super-order-type homogeneous subsequences. }
Next, we provide a natural geometric Ramsey-type theorem in $\R^d$
in which the Ramsey function is a  tower of height~$d$. 

Let $T=(p_1,\ldots,p_{d+1})$ be an ordered $(d+1)$-tuple of points
in $\R^d$. We recall that the \emph{sign} (or \emph{orientation})
of $T$ is 
defined as $\sgn\det M$, where the $j$th column of the $(d+1)\times (d+1)$ 
matrix $M$ is $(1,p_{j,1},p_{j,2},\ldots,p_{j,d})$.
 Geometrically, the sign
is $+1$ if the $d$-tuple
of vectors $p_1-p_{d+1},\ldots,p_{d}-p_{d+1}$ forms a positively
oriented basis of $\R^d$, it is $-1$ if it forms a negatively
oriented basis, and it is $0$ if these vectors are linearly dependent.

We call a sequence $(p_1,p_2,\ldots,p_n)$
of points in $\R^d$ in general position \emph{order-type homogeneous}
if all $(d+1)$-tuples $(p_{i_1},\ldots,p_{i_{d+1}})$,
$i_1<\cdots<i_{d+1}$, have the same sign (which is nonzero,
by the general position assumption). Such sequences are of interest
from various points of view: For example, the convex hull of an order-type
homogeneous sequence is combinatorially equivalent to
a cyclic polytope (see, e.g., \cite{z-lp-94} for background). 
They can also be viewed as \emph{discrete Chebyshev systems};
see  \cite{KarlinStudden}, as well as a remark below.

By Ramsey's theorem, every sufficiently long point sequence 
in general position contains an order-type homogeneous 
subsequence of length $n$ (we color every $(d+1)$-tuple by its
sign). Letting $\OT_d(n)$ be the corresponding
Ramsey function, we obtain $\OT_d(n)\le\twr_d(n^C)$ from Theorem~\ref{t:ub}.
This has recently been improved to
$\OT_d(n)\le\twr_d(O(n))$ by Suk~\cite{Suk-OT}.

This upper bound is essentially tight. Until recently this was proved only for
$d=2$ (by \cite{es-cpg-35}) and $d=3$ \cite{highes-aim}. 
As will be explained next, our results, together with
a recent paper of
B\'ar\'any, P\'or, and the second author \cite{curve-bmp},
 yield a matching lower bound for all~$d$.

In the present paper we prove a lower bound for a somewhat stronger notion
of homogeneity. Namely, let  $\pi_j\colon \R^d \to \R^j$
denote the projection on the first $j$ coordinates. 
We say that a point sequence
$P=(p_1, \dotsc, p_n)$ in
$\R^d$ is \emph{super-order-type homogeneous} if,
for each $j=1,2,\ldots,d$, the
projected sequence $\pi_j(P)=(\pi_j(p_1),\ldots\pi_j(p_n))$
is order-type homogeneous. 

By iterated application of Ramsey's theorem, it can be seen
that every sufficiently long point sequence in general position in $\R^d$
contains a super-order-type homogeneous subsequence of length~$n$.
Let $\COT_d(n)$ be the corresponding Ramsey function. We have the following
lower bound, proved in Section~\ref{s:lb-sup}:

\begin{theorem}\label{t:super}
For every $n\ge d+1$,   $\COT_d(n)\ge \twr_d(n-d)$.
\end{theorem}

In $\cite{curve-bmp}$ it is proved that
$\COT_d(n)\le \OT_d(C_d n)$ for every $d$, where $C_d$ is a suitable
constant. Thus, we also obtain a lower bound for
$\OT_d$, which is tight up to a multiplicative constant
in front of $n$:
\begin{corol}
We have $\OT_d(n) \geq \twr_d(\Omega(n))$.
\end{corol}


\heading{Chebyshev systems. } Let $A$ be a linearly ordered set
of at least $k+1$ elements. A (real) \emph{Chebyshev system} (also spelled
Tchebycheff) on $A$ is a system of continuous real functions
 $f_0,f_1,\ldots,f_k\:A\to\R$ such that for every choice of elements
$t_0<t_1<\cdots<t_k$ in $A$, the matrix $(f_i(t_j))_{i,j=0}^k$ has
a (strictly) positive determinant. Chebyshev systems are mostly
considered for $A$ an interval in $\R$ with the natural ordering, 
the basic example being $f_i(t)=t^i$, but the case
of finite $A$ (\emph{discrete} Chebyshev systems) has been investigated
as well. The functions $f_0,\ldots,f_k$ as above form a
\emph{Markov system}, also called a \emph{complete Chebyshev system},
if $f_0,\ldots,f_i$ is a Chebyshev system for every $i=1,2,\ldots,k$.
Chebyshev systems are of considerable importance in several areas,
such as approximation theory or the theory of finite moments;
see the classical monograph of Karlin and Studden 
\cite{KarlinStudden} or, e.g., Carnicer, Pe\~na, and Zalik
\cite{CPZ} for a more recent study. 

In our setting, it is easy to check that
an $n$-point order-type homogeneous sequence
$P=(p_1,\ldots,p_n)$ in $\R^d$ gives rise to a Chebyshev system
on $A=\{1,2,\ldots,n\}$, by setting $f_j(i)=p_{i,j}$
for $j=1,2,\ldots,d$ and $f_0\equiv 1$ (possibly with changing
the sign for one of the $f_i$, if
the signs of the $(d+1)$-tuples in $P$ are negative),
and conversely, from a discrete Chebyshev system with
$f_0\equiv 1$ we obtain an order-type homogeneous sequence.
Similarly, super-order-type homogeneous sequences correspond
to discrete Markov systems.

\section{Lower bound for semialgebraic predicates in a small dimension}
\label{s:lower}

Here we prove Theorem \ref{t:lb}. As was remarked in the introduction,
our construction can be regarded as a modification of that of
Conlon et al.~\cite{cfpss-semialg}, but we give a self-contained
presentation.

\heading{Stepping up. } The proof proceeds by induction on $d$;
having constructed a suitable $d$-dimensional
 $k$-ary semialgebraic predicate 
and an $N$-point sequence $P\subset \R^d$ without long $\Phi$-homogeneous
subsequences, we produce a $(d+1)$-dimensional
$(k+1)$-ary semialgebraic predicate $\Psi$
and a $2^N$-point sequence $Q\subset\R^{d+1}$ without
long $\Psi$-homogeneous subsequences. 

Our basic tool is a classical stepping-up lemma of Erd\H{o}s and
Hajnal, see e.g.~\cite{grs-rt-90} or \cite{conlon-al}.
We first recall it in the standard combinatorial
setting, and then we will work on transferring it to a semialgebraic
setting.

Let $I=[N]:=\{1,2,\ldots,N\}$, and let $\chi\:\binom{I}{k}\to \{0,1\}$
be a given two-coloring of all $k$-tuples of $I$. Let $J=\{0,1\}^N$
be the set of all binary vectors of length $N$ ordered lexicographically.

We define a coloring $\chi'\:\binom{J}{k+1}\to\{0,1\}$ of all $(k+1)$-tuples
of $J$. First we introduce a function $\delta\:J\times J\to I$
by 
\[\delta(\balpha,\bbeta)=\min\{i\in I:\alpha_i\ne\beta_i\}.
\]
For a $(k+1)$-tuple $(\balpha_1,\ldots,\balpha_{k+1})$
of binary vectors, $\balpha_1<_{\rm lex}\cdots
<_{\rm lex} \balpha_{k+1}$, we write 
$\delta_\ell:=\delta(\balpha_\ell,\balpha_{\ell+1})$.
Then $\chi'$, the \emph{stepping-up coloring} for $\chi$, is given by
\begin{equation}\label{def-stepup}
\chi'(\balpha_1,\ldots,\balpha_{k+1}):=\alterdef{
\chi(\delta_1,\ldots,\delta_k)& \text{if }
        \delta_1 < \dotsb < \delta_k \text{ or } \delta_1 > \dotsb >
        \delta_k\\
1& \text{if } \delta_1 < \delta_2 > \delta_3\\
0& \text{otherwise.}
}
\end{equation}

Now the stepping-up lemma can be stated as follows.

\begin{lemma}[Stepping-up lemma]\label{lem:step-up}
If $\chi$ is a two-coloring of the $k$-tuples of  $I:=[N]$
under which $I$ has no homogeneous subset of size $n$,
then, under the stepping-up coloring $\chi'$,  the set $J=\{0,1\}^N$ contains
no homogeneous subset of size $2n+k-4$.
\end{lemma}

The proof is not very complicated and it can be found, e.g., in
\cite{cfpss-semialg} or \cite[Sec.~4.7]{grs-rt-90}.

\heading{Semialgebraic stepping up. } Now let $\Phi$ be a $d$-dimensional
$k$-ary
semialgebraic predicate, and let $P=(p_1,\ldots,p_N)$
be a point sequence in $\R^d$ indexed by the set $I=[N]$ as above.
Let $\chi=\chi_\Phi$ be the coloring of $k$-tuples of $I$ induced
by $\Phi$; that is, for $i_1<\cdots<i_k\in I$, 
$\chi(i_1,\ldots,i_k)$ is $1$ or $0$ depending
on whether $\Phi(p_{i_1},\ldots,p_{i_k})$ holds or not.

We want to construct a sequence $Q$ in $\R^{d+1}$ indexed by $J=\{0,1\}^N$
and a $(d+1)$-dimensional $(k+1)$-ary semialgebraic predicate $\Psi$ 
such that the coloring induced
by $\Psi$ on $\binom{J}{k+1}$ is exactly the stepping-up coloring $\chi'$.
For our construction, we need to assume simple additional properties
of $\Phi$ and $P$, which we now introduce.

Let $P=(p_1,\ldots,p_N)$ be a sequence of points in $\R^d$.
We call a predicate $\Phi$ \emph{robust}\footnote{Conlon et 
al.~\cite{cfpss-semialg}
use the term $\eta$-deep.}  on $P$
if there is some $\eta>0$ such that
$\Phi(p_{i_1},\ldots,p_{i_k})\Leftrightarrow \Phi(p'_{i_1},\ldots,p'_{i_k})$
whenever $1\le i_1<\cdots<i_k\le N$ and $\|p_{i_j}-p'_{i_j}\|\le\eta$
for all $j=1,2,\ldots,k$.  

In defining the new predicate $\Psi$, we will also need to
use the linear ordering of the points of $P$.
We thus say that a binary semialgebraic predicate $\prec$
on $\R^d$ is \emph{order-inducing} for $P$ if
$p_i\prec p_j$ iff $i<j$, for $i,j=1,2,\ldots,N$.

Now we can state our semialgebraic stepping-up lemma.

%

\begin{prop}[Semialgebraic stepping-up]\label{p:lb-i}
Let $\Phi$ be a $d$-dimensional $k$-ary semialgebraic predicate  and let
$\prec$ be a $d$-dimensional binary semialgebraic predicate. Then
there are a $(d+1)$-dimensional $(k+1)$-ary semialgebraic predicate 
$\Psi$ and a $(d+1)$-dimensional binary 
semialgebraic predicate $\prec'$ with the following
property. 

Let $P=(p_1,\ldots,p_N)$ be a point sequence in $\R^d$ such that
$\prec$ is order-inducing on $P$ and both $\Phi$ and $\prec$
are robust on $P$, and let $\chi_\Phi$ be the coloring
of $k$-tuples of $I=[N]$ induced by $\Phi$. 
Then there is a point sequence $Q=(q_\balpha:\balpha\in J=\{0,1\}^N)$
such that $\prec'$ is order-inducing on $Q$ (w.r.t.\ the lexicographic
ordering of $J$), both $\Psi$ and $\prec'$ are robust on $Q$,
and the coloring $\chi_\Psi$ induced on the $(k+1)$-tuples
of $J$ by $\Psi$ is the stepping-up coloring for~$\chi_\Phi$.
\end{prop}

\begin{proof}
The construction of $Q$ uses a parameter $\eps>0$, which we
assume to be sufficiently small.

For $\balpha=(\alpha_1,\ldots,\alpha_N)\in J$, we set
\[
q_\balpha := \sum_{i=1}^N \alpha_i\eps^i (1,p_{i,1},p_{i,2},\ldots,p_{i,d})\in \R^{d+1}.
\]

In particular, the first coordinate of $q_{\balpha}$ is
$\sum_{i=1}^N \alpha_i\eps^i$. Hence, as is easy to check,
for $\eps$ sufficiently small, 
the lexicographic ordering
of $J$ agrees with the ordering of $Q$ by the first coordinate,
and hence we can take the standard ordering in the first coordinate
as the required order-inducing (and obviously robust)
predicate $\prec'$ on~$Q$.

Next, we define a mapping $\sdelta\:\R^{d+1}\times\R^{d+1}\to \R^d$,
which will play the role of the $\delta$ from the stepping-up lemma
in the geometric setting. For points $x,y\in \R^{d+1}$, we set
\begin{equation}\label{e:sdelta}
\sdelta(x,y):= \left(\frac{x_2-y_2}{x_1-y_1},\frac{x_3-y_3}{x_1-y_1},\ldots,
\frac{x_{d+1}-y_{d+1}}{x_1-y_1}\right) \in\R^d.
\end{equation}
(Actually, $\sdelta(x,y)$ is undefined for $x_1=y_1$, but we will 
use $\sdelta$ only for points with different first coordinates.)

By elementary calculation we can see that for $\balpha,\bbeta\in J$,
$\balpha\ne\bbeta$, we have
\begin{equation}
\label{e:lim}
\lim_{\eps\to 0} \sdelta(q_\balpha,q_{\bbeta})= p_{\delta(\balpha,\bbeta)}.
\end{equation}
This allows us to imitate the combinatorial definition \eqref{def-stepup}
of the stepping-up coloring
by a semialgebraic predicate $\Psi$. For a $(k+1)$-tuple of points
$(x_1,\ldots,x_{k+1})$ in $\R^{d+1}$, let us write
$\sdelta_\ell:=\sdelta(x_\ell,x_{\ell+1})$, and set
\[
\Psi(x_1,\ldots,x_{k+1}):=\alterdef{
\Phi(\sdelta_1,\ldots, \sdelta_k)&\mbox{ if }
\sdelta_1 \prec \cdots \prec \sdelta_k\\
\Phi(\sdelta_k,\ldots, \sdelta_1)&\mbox{ if }
\sdelta_1 \succ \cdots \succ \sdelta_k\\
\mathrm{true}&\mbox{ if } \sdelta_1 \prec \sdelta_2\succ \sdelta_3\\
\mathrm{false}&\mbox{ otherwise.}
}
\]
As written, $\Psi$ is not necessarily a semialgebraic predicate,
since the definition of $\sdelta$ involves division. However, we can
always multiply by the denominators and introduce appropriate
conditions; e.g., $\frac uv<1$ can be replaced with
$(u<v\wedge v>0)\vee (u>v\wedge v<0)$, which is equivalent whenever $\frac uv$
is defined. In this way, we obtain an honest semialgebraic predicate.

It remains to check that $\Psi$ induces the stepping-up coloring
on $J$, which is straightforward using the robustness of $\Phi$
and $\prec$ and the limit relation \eqref{e:lim}. Indeed,
let us fix  $\balpha_1<_{\rm lex}\cdots <_{\rm lex}\balpha_{k+1}\in J$ and
write $\sdelta_\ell:=\sdelta(q_{\balpha_\ell},q_{\balpha_{\ell+1}})$
and $\delta_\ell:=\delta(\balpha_\ell,\balpha_{\ell+1})$.
Then for $\eps$ sufficiently small, we have $\sdelta_\ell
\prec\sdelta_{\ell+1}$ iff $p_{\delta_\ell}\prec p_{\delta_{\ell+1}}$
(by the robustness of $\prec$) iff $\delta_\ell<\delta_{\ell+1}$
(since $\prec$ is order-inducing on $P$).
Assuming $\sdelta_1\prec \sdelta_2\prec\cdots\prec\sdelta_k$,
we get that $\Phi(\sdelta_1,\ldots,\sdelta_k)$ iff
$\Phi(p_{\delta_1},\ldots,p_{\delta_k})$, again for
all sufficiently small $\eps$; similarly if
$\sdelta_1\succ \sdelta_2\succ\cdots\succ\sdelta_k$.
Therefore,
the coloring induced by $\Psi$ on $Q$ is indeed
the stepping-up coloring for $\chi_\Phi$ as claimed.

It remains to verify that $\Psi$ is robust on $Q$, but this is clear
from the robustness of $\Phi$ and $\prec$ and the continuity of 
$\sdelta$ on the subset of $\R^{d+1}\times\R^{d+1}$ where it is defined.
\end{proof}


\heading{Proof of Theorem \ref{t:lb}.}
As announced, we prove the theorem by induction on $d$.

For the base case $d=1$, we use a result of Conlon et al.~\cite{cfpss-semialg},
who construct a $4$-ary semialgebraic predicate $\Phi_1$ on $\R^1$
and, for every $n$, a sequence $P_1\subset \R$ of length $\twr_3(\Omega(n))$
with no $\Psi_1$-homogeneous subsequence of length $n$. It is obvious
from their construction that $\Psi_1$ is robust on $P_1$ and
that $<$, the usual inequality among real numbers, is robust and order-inducing
on $P_1$.

The theorem then follows  by a $(d-1)$-fold application of 
Proposition~\ref{p:lb-i} together with the stepping-up 
lemma (Lemma~\ref{lem:step-up}).
\proofend

\section{Lower bound for super order type}\label{s:lb-sup}

Here we prove Theorem~\ref{t:super}. Thus, we need to exhibit
long point sequences without super-order-type homogeneous subsequences
of length $n$. The construction is almost identical to the one
in the previous section, only the base case for $d=1$ is different.
The proof essentially consists in relating super-order-type homogeneity
to another property, which we call super-monotonicity; checking
that the constructed sequence has no super-monotone subsequences
of length~$n$
is straightforward. 


First, for convenience, we extend the definition of the bivariate function
$\sdelta$ from \eqref{e:sdelta} in
the previous section to an arbitrary number of arguments.
Namely, we set $\sdelta(p)=p$ and, for $k\ge 2$,
\[
\sdelta(p_1,\dots,p_{k+1}):=\sdelta(\sdelta(p_1,\dots,p_k),\sdelta(p_2,\dots,p_{k+1})).
\]
Again, we are going to use $\sdelta$ only with arguments for which it
is well defined.

For points $p,q\in\R^d$, we write $p\leqone q$ if $p_1<q_1$ (strict
inequality in the first coordinate).
A point sequence $P=(p_1\dots,p_n)$ in $\R^d$ is 
\emph{super-monotone} if each of the point sequences 
$(\sdelta(p_1,\dots,p_j),\dots,\sdelta(p_{n-j+1},\dots,p_n))$ 
in $\R^{d-j+1}$  is monotone according to $\leqone$,  $1\le j\le d$.

Here is the key technical result.

\begin{prop}\label{p:sot-sm}
A point sequence $(p_1,\dots,p_n)$ in $\R^d$ is super-monotone 
if and only if it is super-order-type homogeneous.
\end{prop}

The proof will be given at the end of this section, after
some algebraic lemmas. First we finish the proof of Theorem~\ref{t:super},
assuming the proposition.

\heading{Proof of Theorem \ref{t:super}.}
We will construct a sequence $P_d(n)$ in general position
in $\R^d$ of length $\twr_d(n-d)$ and containing 
no super-order-type homogeneous subsequence of length~$n$.

We proceed by induction on $d$. The inductive hypothesis will include
the assumption that the first coordinate in $P_d(n)$ is strictly
increasing.

For $d=1$ we set $P_1(n):=(1,2,\dots,n-1)$. 

Now we construct $P_{d+1}(n)$ from $P_d(n-1)=(p_1,\ldots,p_N)$,
using the same construction as in Proposition \ref{p:lb-i}. That is,
$
P_{d+1}(n)=(q_\balpha:\balpha\in \{0,1\}^N),
$
where the binary vectors $\balpha$ are ordered lexicographically,
and where, with $\eps>0$ sufficiently small,
\[
q_\balpha:= \sum_{i=1}^N \alpha_i\eps^i(1,p_{i,1},p_{i,2},\ldots,p_{i,d}) \in \R^{d+1}, \quad \balpha\in\{0,1\}^N.
\]
(The $\eps$ is different in each inductive step, and in particular,
the one used to construct $P_{d+1}(n)$ from $P_d(n-1)$ is much smaller
than the one used to construct $P_{d}(n-1)$ from $P_{d-1}(n-2)$, etc.)
Because of the robustness of the super-order-type condition, we can slightly perturb the points so that they are in general position.
As in the previous section, the points of $P_{d+1}(n)$, ordered
according to the lexicographic ordering
of the indices $\balpha$, have increasing first coordinates
 (for $\eps$ sufficiently small).

Now we assume for contradiction
 that $P_{d+1}(n)$ contains a super-order-type homogeneous subsequence 
$S=(s_1,\dots,s_n)$.
By Proposition~\ref{p:sot-sm}, $S$ is super-monotone. 
Thus, setting $t_\ell=\sdelta(s_\ell,s_{\ell+1})$, $\ell=1,2,\ldots,n-1$,
the sequence $T=(t_1,\ldots,t_{n-1})$ is super-monotone as well by definition.

By the limit relation \eqref{e:lim}, for $\eps\to 0$,
each $t_\ell$ tends to a point $p_{i_\ell}$ of $P_{d}(n-1)$.
Moreover, by super-monotonicity, we have $t_1\leqone\cdots\leqone t_{n-1}$.
Hence $p_{i_1}\leqone \cdots\leqone p_{i_{n-1}}$ for sufficiently
small $\eps$ and therefore, since
the first coordinates are increasing in $P_d(n-1)$ by the inductive
hypothesis, we have $i_1<\cdots<i_{n-1}$.
Consequently, using Proposition~\ref{p:sot-sm} again, $(p_{i_1}, \dotsc,
p_{i_{n-1}})$ is a super-order-type homogeneous subsequence of
$P_d(n-1)$---a contradiction proving the theorem.
\proofend

\heading{Algebraic lemmas. } It remains to prove Proposition~\ref{p:sot-sm},
and for this, we need to develop some algebraic results.

Given a $k$-tuple $T=(p_1,\dots,p_k)$ of points in $\R^d$, $1\le k\le d$,
 and an index $j\ge k-1$, we put
\[
D_j(T)=\det
\begin{pmatrix}
1 & 1 & \dots & 1\\
p_{1,1} & p_{2,1} & \dots & p_{k,1}\\
\vdots & \vdots & \ddots & \vdots\\
p_{1,k-2} & p_{2,k-2} & \dots & p_{k,k-2}\\
p_{1,j} & p_{2,j} & \dots & p_{k,j}\\
\end{pmatrix}
\]
and
\[\overrightarrow D_j(T)=(D_j(T),D_{j+1}(T),\dots,D_d(T)).\]
Let us remark that $k$ is not represented explicitly in the notation,
but it can be inferred from the number of arguments of $D_j$.
We also note that $\sgn D_{k-1}(p_1,\dots,p_k)$ is the sign of the $k$-tuple
$\pi_{k-1}(T)$.


\begin{lemma}\label{lem:matrices}
If $A=(p_1,\dots,p_k)$ and $B=(p_2,\dots,p_{k+1})$, then, for
$j\ge k$, we have
\[D_{k-1}(A) D_{j} (B)- D_{k-1}(B) D_{j}(A)=D_{k-2}(p_2,\dots,p_k) D_{j}(p_1,\dots,p_{k+1}).\]
\end{lemma}

\begin{proof}
It is enough to do the case $j=k$ (we have $j\ge k$, and so
in the identity of the lemma, the
$j$th coordinates of the $p_i$ appear only in the determinants
$D_j(A)$, $D_j(B)$, and $D_j(p_1,\ldots,p_{k+1})$).
We define the $(k+1)\times(k+1)$ matrix
\[M_{k+1}(p_1,\dots,p_{k+1})=
\begin{pmatrix}
1 & 1 & \dots & 1\\
p_{1,1} & p_{2,1} & \dots & p_{k+1,1}\\
\vdots & \vdots & \ddots & \vdots\\
p_{1,k} & p_{2,k} & \dots & p_{k+1,k}
\end{pmatrix}.\]

All the determinants we are interested in are submatrices 
of $M_{k+1}$ and they all contain the matrix $M_{k-1}=
M_{k-1}(p_2,\dots,p_k)$ associated with $D_{k-2}(p_2,\dots,p_k)$.
We can use elementary row and column operations on $M_{k+1}$ to diagonalize 
$M_{k-1}$ while leaving the determinants fixed, and we can also assume 
that the entries below $M_{k-1}$, as well as those
to the left and to the right of it, are~$0$, as is illustrated next:
\[
\begin{pmatrix}
\cline{2-4}
1 & \multicolumn{1}{|c}{1} & \dots & \multicolumn{1}{c|}{1} & 1\\
p_{1,1} & \multicolumn{1}{|c}{p_{2,1}} & \dots & \multicolumn{1}{c|}{p_{k,1}} & p_{k+1,1}\\
\vdots & \multicolumn{1}{|c}{\vdots} & M_{k-1} & \multicolumn{1}{c|}{\vdots} & \vdots\\
p_{1,k-2} & \multicolumn{1}{|c}{p_{2,k-2}} & \dots & \multicolumn{1}{c|}{p_{k,k-2}} & p_{k+1,k-2}\\
\cline{2-4}
p_{1,k-1} & p_{2,k-1} & \dots & p_{k+1,k-1} & p_{k+1,k-1}\\
p_{1,k} & p_{2,k} & \dots & p_{k+1,k} & p_{k+1,k}\\
\end{pmatrix}
\longrightarrow
\begin{pmatrix}
\cline{2-4}
0 & \multicolumn{1}{|c}{m_1} & \dots & \multicolumn{1}{c|}{0} & 0\\
0 & \multicolumn{1}{|c}{0} & \dots & \multicolumn{1}{c|}{0} & 0\\
\vdots & \multicolumn{1}{|c}{\vdots} & \ddots & \multicolumn{1}{c|}{\vdots} & \vdots\\
0 & \multicolumn{1}{|c}{0} & \dots & \multicolumn{1}{c|}{m_{k-2}} & 0\\
\cline{2-4}
x & 0 & \dots & 0 & u\\
y & 0 & \dots & 0 & v\\
\end{pmatrix}.
\]
Now we can compute the determinants in the following way:
\[
\begin{aligned}
D_{k}(p_1,\ldots,p_{k+1})&=(-1)^{k+1}(xv-yu)\det(M_{k-1})\\
D_{k-2}(p_2,\dots,p_k)&=\det(M_{k-1})\\
D_{k-1}(A)&=(-1)^{k+1}x\det(M_{k-1})
\end{aligned}
\qquad
\begin{aligned}
D_{k}(B)&=v\det(M_{k-1})\\
D_{k}(A)&=(-1)^{k+1}y\det(M_{k-1})\\
D_{k-1}(B)&=u\det(M_{k-1}).
\end{aligned}
\]
The lemma follows.
\end{proof}

\begin{lemma}\label{lem:sigma}
\[\sdelta(p_1,\dots,p_k)=\frac{\overrightarrow{D}_k(p_1,\dots,p_k)}{D_{k-1}(p_1,\dots,p_k)}.\]
\end{lemma}

\begin{proof}
The proof goes by induction on $k$. The cases $k=1,2$ are trivial.
Assume the lemma is true for $k$ and we have points 
$p_1,\dots,p_{k+1}\in\R^d$. For simplicity we write
 $A=(p_1,\dots,p_k)$ and $B=(p_2,\dots,p_{k+1})$. Then we have,
with $\pi\:\R^d\to\R^{d-1}$ denoting the projection omitting the
first coordinate,
\begin{align*}
\sdelta(p_1,\dots,p_{k+1})&=\sdelta(\sdelta(A),\sdelta(B))\\
&= \frac{\pi(\sdelta(A))-\pi(\sdelta(B))}{(\sdelta(A))_1-(\sdelta(B))_1}\\
&= \frac{D_{k-1} (A)\overrightarrow{D}_{k+1} (B)-D_{k-1} (B)\overrightarrow{D}_{k+1} (A)}
{D_{k-1} (A)D_{k} (B)-D_{k-1} (B)D_{k} (A)}.
\end{align*}
The last equality follows from the fact that $\pi(\overrightarrow{D}_{k}) = \overrightarrow{D}_{k+1}$ and by clearing the denominators.
To finish the proof we use Lemma~\ref{lem:matrices} on the denominator and each coordinate of the numerator.
\end{proof}


\heading{Proof of Proposition~\ref{p:sot-sm}.}
We generalize the notions of super-monotonicity and
super-order-type homogeneity as follows.
We say that a point sequence $(p_1\dots,p_n)$ is \emph{$k$-monotone} 
if for all $j\le k$ the point sequence $(\sdelta(p_1,\dots,p_j),\dots,\sdelta(p_{n-j+1},\dots,p_n))$ is monotone according to $\leqone$.
We say that $(p_1\dots,p_n)$ is \emph{$k$-order-type homogeneous} if for all $j\le k$ the sequence of projections $(\pi_j(p_1)\dots,\pi_j(p_n))$ 
in $\R^j$ is order-type homogeneous.

By induction on $k$, we prove that
a point sequence $(p_1,\dots,p_n)$ in $\R^d$ is $k$-monotone 
if and only if it is $k$-order-type homogeneous; for $k=d$
this is the statement of the proposition.

The cases $k=1,2$ are trivial. So we assume that the claim
 is true up to some $k$ and we are given a sequence of $n$ points.
We may assume that this sequence is $k$-monotone, and hence also 
$k$-order-type homogeneous. Then we only need to show that
$(\sdelta(p_1,\dots,p_{k+1}),\dots,\sdelta(p_{n-k},\dots,p_n))$ is 
order-type homogeneous iff $(\pi_{k+1}(p_1),\dots,\pi_{k+1}(p_n))$ 
is monotone according to~$\leqone$.

Let $(q_1,\ldots,q_{k+2})$ be a $(k+2)$-point subsequence of $(p_1,\ldots,p_n)$.
By Lemma~\ref{lem:sigma}, the condition $\sdelta(q_1,\dots,q_{k+1})\leqone\sdelta(q_2,\dots,q_{k+2})$ is equivalent to
\begin{equation}\label{eq:1}
\frac{D_{k+1}(q_2,\dots,q_{k+2})} {D_{k}(q_2,\dots,q_{k+2})}-\frac{D_{k+1}(q_1,\dots,q_{k+1})} {D_{k}(q_1,\dots,q_{k+1})} >0.
\end{equation}
Since $(q_1,\ldots,q_{k+2})$ is $k$-order-type homogeneous,  we have
\[D_{k}(q_1,\dots,q_{k+1})D_{k}(q_2,\dots,q_{k+2})>0,\]
and therefore, \eqref{eq:1} is equivalent to
\[D_{k+1}(q_2,\dots,q_{k+2})D_{k}(q_1,\dots,q_{k+1})-D_{k+1}(q_1,\dots,q_{k+1})D_{k}(q_2,\dots,q_{k+2})>0.\]
By Lemma~\ref{lem:matrices} this is just
\[D_{k-1}(q_2\dots,q_{k+1})D_{k+1}(q_1\dots,q_{k+2})>0.\]
Since our sequence is also $(k-1)$-order-type homogeneous, the numbers
\[D_{k-1}(p_1\dots,p_{k}),\,D_{k-1}(p_2\dots,p_{k+1}),\,\ldots,\,D_{k-1}(p_{n-k+1}\dots,p_{n})\]
have the same sign and therefore the numbers
\[D_{k+1}(p_1\dots,p_{k+2}),\,D_{k+1}(p_2\dots,p_{k+3}),\,\ldots,\,D_{k+1}(p_{n-k-1}\dots,p_{n})\]
also have the same sign.
This is precisely the condition needed for the sequence
to be $(k+1)$-order-type homogeneous.
\proofend

\subsection*{Acknowledgment}

We would like to thank Imre B\'ar\'any for useful discussions.

%
%
%

\bibliographystyle{alpha}
\bibliography{cg,geom}

\end{document}